\documentclass[12pt,a4paper]{amsart}

\usepackage[hmargin=2cm,vmargin=2cm]{geometry}

\usepackage{amsmath,amssymb}
\usepackage{tikz}

\def\mcY{{\mathcal Y}}

\def\mbZ{{\mathbb Z}}
\def\mbC{{\mathbb C}}
\def\mbL{{\mathbb L}}
\def\mbQ{{\mathbb Q}}

\newcommand*{\Hilb}[1]{(\mbC^2)^{[#1]}}
\newcommand*{\torus}[1]{(\mbC^*)^{#1}}

\newtheorem{proposition}{Proposition}[section]
\newtheorem{theorem}[proposition]{Theorem}
\newtheorem{corollary}[proposition]{Corollary}
\newtheorem{lemma}[proposition]{Lemma}

\title[Betti numbers of the quasihomogeneous Hilbert schemes]{A simple proof of the formula for the Betti numbers of the quasihomogeneous Hilbert schemes.}

\author{Alexandr Buryak}

\address{Alexandr~Buryak:\newline
Department of Mathematics,
ETH Zurich, Ramistrasse 101 8092, HG G27.1, Zurich, Switzerland.} 
\email{buryaksh@gmail.com}

\author{Boris Lvovich Feigin}

\address{Boris~Lvovich~Feigin:\newline
National Research University Higher School of Economics, Russia, Moscow, 101000, Myasnitskaya ul., 20, \newline
Landau Institute for Theoretical Physics, Russia, Chernogolovka, 142432, prosp. Akademika Semenova, 1a, and \newline
Independent University of Moscow, Russia, Moscow, 119002, Bolshoy Vlasyevskiy per., 11}
\email{borfeigin@gmail.com}

\author{Hiraku Nakajima}

\address{Hiraku~Nakajima:\newline
Research Institute for Mathematical Sciences, Kyoto University, Kyoto 606-8502, Japan}
\email{nakajima@kurims.kyoto-u.ac.jp}

\begin{document}

\begin{abstract}
In a recent paper the first two authors proved that the generating series of the Poincare polynomials of the quasihomogeneous Hilbert schemes of points in the plane has a simple decomposition in an infinite product. In this paper we give a very short geometrical proof of that formula.   
\end{abstract}

\maketitle

\section{INTRODUCTION}

The Hilbert scheme $(\mbC^2)^{[n]}$ of $n$ points in the plane $\mbC^2$ parametrizes ideals $I\subset\mbC[x,y]$ of colength $n$: $\dim_{\mbC}\mbC[x,y]/I=n$. It is a nonsingular, irreducible, quasiprojective algebraic variety of dimension $2n$ with a rich and much studied geometry, see \cite{Got,Nak1} for an introduction.

The cohomology groups of $(\mbC^2)^{[n]}$ were computed in \cite{EllStr}, and the ring structure in the cohomology was determined independently in the papers \cite{LehSor} and \cite{Vas}.

There is a $(\mbC^*)^2$-action on $(\mbC^2)^{[n]}$ that plays a central role in this subject. The algebraic torus $(\mbC^*)^2$ acts on $\mbC^2$ by scaling the coordinates, $(t_1,t_2)\cdot(x,y)=(t_1x,t_2y)$. This action lifts to the $(\mbC^*)^2$-action on the Hilbert scheme $(\mbC^2)^{[n]}$.

For arbitrary non-negative integers $\alpha$ and $\beta$, such that $\alpha+\beta\ge 1$, let $T_{\alpha,\beta}=\{(t^\alpha,t^\beta)\in(\mbC^*)^2|t\in\mbC^*\}$ be a one-dimensional subtorus of $(\mbC^*)^2$. If $\alpha$ and $\beta$ are non-zero, then the fixed point set $\left(\Hilb{n}\right)^{T_{\alpha,\beta}}$ is called the quasihomogeneous Hilbert scheme of points on the plane $\mbC^2$. 

The quasihomogeneous Hilbert scheme $\left(\Hilb{n}\right)^{T_{\alpha,\beta}}$ is compact and in general has many irreducible components. They were described in \cite{Eva}. In the case $\alpha=1$ the Poincare polynomials of the irreducible components were computed in \cite{Bur}. 

The Poincare polynomial of a manifold $X$ is defined by $P_q(X)=\sum_{i\ge 0}\dim H_i(X;\mbQ)q^{\frac{i}{2}}$. In~\cite{BurFei} the first two authors proved the following theorem. 

\begin{theorem}\label{theorem: quasihomogeneous1}
Suppose $\alpha$ and $\beta$ are positive coprime integers, then
\begin{gather*}
\sum_{n\ge 0}P_q\left(\left((\mbC^2)^{[n]}\right)^{T_{\alpha,\beta}}\right)t^n=\prod_{\substack{i\ge 1\\(\alpha+\beta)\nmid i}}\frac{1}{1-t^i}\prod_{i\ge 1}\frac{1}{1-qt^{(\alpha+\beta)i}}.
\end{gather*}
\end{theorem}

In this paper we give another proof of this theorem. In \cite{BurFei} the large part of the proof consists of non-trivial combinatorial computations with Young diagrams. Our new proof is more geometrical and is much shorter. In fact, we prove a slightly more general statement. 

Let $\Gamma_m$ be the finite subgroup of $(\mbC^*)^2$ defined by 
$$
\Gamma_m=\left\{(\zeta^j,\zeta^{-j})\in\torus{2}\left|\zeta=\exp\left(\frac{2\pi i}{m}\right),j=0,1,\ldots,m-1\right.\right\}.
$$
For a manifold $X$ let $H^{BM}_*(X;\mbQ)$ denote the Borel-Moore homology group of $X$ with rational coefficients. Let $P^{BM}_q(X)=\sum_{i\ge 0}\dim H^{BM}_i(X;\mbQ)q^{\frac{i}{2}}$. 

We prove the following theorem.

\begin{theorem}\label{theorem: quasihomogeneous2}
Let $\alpha$ and $\beta$ be any two non-negative integers, such that $\alpha+\beta\ge 1$. Then we have
\begin{gather}\label{formula: quasihomogeneous2}
\sum_{n\ge 0}P^{BM}_q\left(\left((\mbC^2)^{[n]}\right)^{T_{\alpha,\beta}\times\Gamma_{\alpha+\beta}}\right)t^n=\prod_{\substack{i\ge 1\\(\alpha+\beta)\nmid i}}\frac{1}{1-t^i}\prod_{i\ge 1}\frac{1}{1-qt^{(\alpha+\beta)i}}.
\end{gather}
\end{theorem}
Here we use Borel-Moore homology, because the variety $\left((\mbC^2)^{[n]}\right)^{T_{\alpha,\beta}\times\Gamma_{\alpha+\beta}}$ is in general not compact, if $\alpha=0$.

If $\alpha$ and $\beta$ are coprime, then $\Gamma_{\alpha+\beta}\subset T_{\alpha,\beta}$. Hence, Theorem~\ref{theorem: quasihomogeneous1} follows from Theorem~\ref{theorem: quasihomogeneous2}.

Our proof of Theorem~\ref{theorem: quasihomogeneous2} consists of two steps. First, we prove that the left-hand side of~\eqref{formula: quasihomogeneous2} depends only on the sum $\alpha+\beta$. We use an argument with an equivariant symplectic form that is very similar to the one that was applied by the third author in \cite{Nak2} (proof of Proposition~5.7). After that the case $\alpha=0$ can be done using a notion of a power structure over the Grothendieck ring of quasiprojective varieties.

In \cite{BurFei}, as a corollary of Theorem~\ref{theorem: quasihomogeneous1}, there was derived a combinatorial identity. In the same way Theorem~\ref{theorem: quasihomogeneous2} leads to a more general combinatorial identity. Denote by $\mcY$ the set of all Young diagrams. The number of boxes in a Young diagram $Y$ is denoted by $|Y|$. For a box $s\in Y$ we define the numbers $l_Y(s)$ and $a_Y(s)$, as it is shown on Fig.~\ref{armlegpic}.

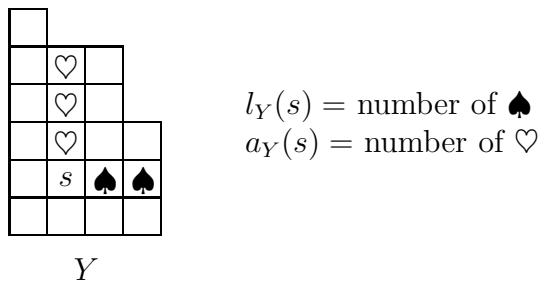
\begin{figure}[h]
\begin{center}
\begin{picture}(70,40)
\put(0,6){
\multiput(0,0)(0,5){7}{\line(1,0){5}}
\multiput(5,0)(0,5){6}{\line(1,0){5}}
\multiput(10,0)(0,5){6}{\line(1,0){5}}
\multiput(15,0)(0,5){4}{\line(1,0){5}}

\multiput(0,0)(5,0){5}{\line(0,1){5}}
\multiput(0,5)(5,0){5}{\line(0,1){5}}
\multiput(0,10)(5,0){5}{\line(0,1){5}}
\multiput(0,15)(5,0){4}{\line(0,1){5}}
\multiput(0,20)(5,0){4}{\line(0,1){5}}
\multiput(0,25)(5,0){2}{\line(0,1){5}}

\put(6.5,6.5){$s$}
\multiput(11,6)(5,0){2}{$\spadesuit$}
\multiput(5.8,10.8)(0,5){3}{$\heartsuit$}

\put(31,11){$a_Y(s)=$ number of $\heartsuit$}
\put(31,16){$l_Y(s)=$ number of $\spadesuit$}
}
\put(8.5,0){$Y$}
\end{picture}
\end{center}
\caption{Definition of the numbers $a_Y(s)$ and $l_Y(s)$}
\label{armlegpic}
\end{figure}

For a Young diagram $Y$ define the number $h_{\alpha,\beta}(Y)$ by
$$
h_{\alpha,\beta}(Y)=\left\{s\in Y\left|\begin{smallmatrix}\alpha l_Y(s)=\beta(a_Y(s)+1)\\ (\alpha+\beta)\mid l_Y(s)+a_Y(s)+1\end{smallmatrix}\right.\right\}.
$$

The following corollary is a generalization of Theorem~1.2 from \cite{BurFei}.
\begin{corollary}\label{theorem: combinatorial identity}
Let $\alpha$ and $\beta$ be arbitrary non-negative integers, such that $\alpha+\beta\ge 1$. Then we have
\begin{gather}\label{formula: generalized identity}
\sum_{Y\in\mcY}q^{h_{\alpha,\beta}(Y)}t^{|Y|}=\prod_{\substack{i\ge 1\\(\alpha+\beta)\nmid i}}\frac{1}{1-t^i}\prod_{i\ge 1}\frac{1}{1-qt^{(\alpha+\beta)i}}.
\end{gather}
\end{corollary}
\begin{proof}
The proof is similar to the proof of Theorem~1.2 in \cite{BurFei}. We apply the results from \cite{B-B1,B-B2}, in order to construct a cell decomposition of the variety $\left(\Hilb{n}\right)^{T_{\alpha,\beta}\times\Gamma_{\alpha+\beta}}$, and show that the left-hand side of~\eqref{formula: quasihomogeneous2} is equal to the left-hand side of~\eqref{formula: generalized identity}.
\end{proof}
We thank Ole Warnaar for suggesting this more general combinatorial identity after the paper \cite{BurFei} was published in arXiv. 

\subsection{\it Organization of the paper} 

In Section~\ref{section: power structure} we recall the definition of the Grothendieck ring of complex quasiprojective varieties and the properties of the natural power structure over it. Section~\ref{section: proof} contains the proof of Theorem~\ref{theorem: quasihomogeneous2}.


\section{POWER STRUCTURE OVER THE GROTHENDIECK RING $K_0(\nu_{\mbC})$}\label{section: power structure}

In this section we review the definition of the Grothendieck ring of complex quasiprojective varieties and the power structure over it. 

\subsection{\it Grothendieck ring}

The Grothendieck ring $K_0(\nu_{\mbC})$ of complex quasiprojective varieties is the abelian group generated by the classes $[X]$ of all complex quasiprojective varieties $X$ modulo the relations:
\begin{enumerate}
\item if varieties $X$ and $Y$ are isomorphic, then $[X]=[Y]$;
\item if $Y$ is a Zariski closed subvariety of $X$, then $[X]=[Y]+[X\backslash Y]$.
\end{enumerate}  
The multiplication in $K_0(\nu_{\mbC})$ is defined by the Cartesian product of varieties: $[X_1]\cdot[X_2]=[X_1\times X_2]$. The class $\left[\mathbb A^1_{\mbC}\right]\in K_0(\nu_{\mbC})$ of the complex affine line is denoted by $\mbL$.

We will need the following property of the Grothendieck ring $K_0(\nu_{\mbC})$. There is a natural homomorphism $\theta\colon\mbZ[z]\to K_0(\nu_{\mbC})$, defined by $z\mapsto\mbL$. This homomorphism is injective (see~e.g.\cite{Loo}). 

\subsection{\it Power structure}

In \cite{GusLueMel} there was defined a notion of a power structure over a ring and there was described a natural power structure over the Grothendieck ring $K_0(\nu_\mbC)$. This means that for a series $A(t)=1+a_1t+a_2t^2+\ldots\in 1 + t\cdot K_0(\nu_\mbC)[[t]]$ and for an element $m\in K_0(\nu_\mbC)$ one defines a series $(A(t))^m \in 1 + t\cdot K_0(\nu_\mbC)[[t]]$, so that all the usual properties of the exponential function hold. 

For a series $A(t)=(1-t)^{-1}$ and a quasiprojective variety $M$ the series $A(t)^{[M]}$ coincides with the motivic zeta function~$\zeta_{[M]}(t)$ introduced by M. Kapranov (\cite{Kap}):
$$
(1-t)^{-[M]}=1+\sum_{i\ge 1}[S^i M]t^i,
$$
where $S^i M$ is the $i$-th symmetric product of the variety $M$. There is the following generalization of this formula. Suppose that $M_1,M_2,\ldots$ and $N$ are quasiprojective 
varieties. Then we have
\begin{align}
&\left(1+\sum_{i\ge 1}[M_i]t^i\right)^{[N]}=1+\sum_{n\ge 1}X_n t^n,\quad\text{where}\notag\\
&X_n=\sum_{\sum_{i\ge 1} i d_i=n}\left[\left.\left(\left(N^{\sum d_i}\backslash\Delta\right)\times\left(\prod M_i^{d_i}\right)\right)\right/\prod S_{d_i}\right].\label{formula: classformula}
\end{align}
Here $\Delta$ is the ''large diagonal'' in $N^{\sum d_i}$, which consists of $\left(\sum d_i\right)$ points of $N$ with at least two coinciding ones. The permutation group $S_{d_i}$ acts by permuting corresponding $d_i$ factors in $\prod N^{d_i}$ and $\prod M_i^{d_i}$ simultaneously.  

We also need the following property of the power structure over $K_0(\nu_\mbC)$. For any $i\ge 1$ and $j\ge 0$ we have \begin{gather}\label{Lpower}
(1-\mbL^j t^i)^{-\mbL}=(1-\mbL^{j+1}t^i)^{-1}.
\end{gather}
It can be derived from several statements from \cite{GusLueMel} as follows. Let $a_i$, $i\ge 1$, and $m$ be from the Grothendieck ring $K_0(\nu_\mbC)$ and $A(t)=1+\sum_{i\ge 1}a_i t^i$. Then for any $s\ge 0$ we have
\begin{align}
&A(\mbL^s t)^m=\left.\left(A(t)^m\right)\right|_{t\mapsto\mbL^s t},\label{statement 2}\\
&(1-t)^{-\mbL^s m}=\left.(1-t)^{-m}\right|_{t\mapsto\mbL^s t}.\label{statement 3}
\end{align}
Formula~\eqref{statement 2} follows from Statement~2 in \cite{GusLueMel} and equation~\eqref{statement 3} follows from Statement~3 in \cite{GusLueMel}. Also for any $s\ge 1$ we have (see \cite{GusLueMel})
\begin{gather}
A(t^s)^m=\left.\left(A(t)^m\right)\right|_{t\mapsto t^s}.\label{tmp}
\end{gather}
Obviously, formula~\eqref{Lpower} follows from \eqref{statement 2}, \eqref{statement 3} and \eqref{tmp}. 


\section{PROOF OF THEOREM \ref{theorem: quasihomogeneous2}}\label{section: proof}

Using the $\torus{2}$-action on $\Hilb{n}$ and the results from \cite{B-B1,B-B2} one can easily construct a cell decomposition of $\left(\Hilb{n}\right)^{\Gamma_{\alpha+\beta}\times T_{\alpha,\beta}}$. Thus, Theorem~\ref{theorem: quasihomogeneous2} is equivalent to the following formula
\begin{gather}\label{formula: class}
\sum_{n\ge 0}\left[\left((\mbC^2)^{[n]}\right)^{T_{\alpha,\beta}\times\Gamma_{\alpha+\beta}}\right]t^n=\prod_{\substack{i\ge 1\\(\alpha+\beta)\nmid i}}\frac{1}{1-t^i}\prod_{i\ge 1}\frac{1}{1-\mbL t^{(\alpha+\beta)i}}.
\end{gather}

It is clear that equation~\eqref{formula: class} is a corollary of the following two lemmas.

\begin{lemma}\label{lemma: first lemma}
For any $\alpha,\beta\ge 0$, such that $\alpha+\beta\ge 1$, we have
$$
\left[\left((\mbC^2)^{[n]}\right)^{T_{\alpha,\beta}\times\Gamma_{\alpha+\beta}}\right]=\left[\left((\mbC^2)^{[n]}\right)^{T_{0,\alpha+\beta}\times\Gamma_{\alpha+\beta}}\right].
$$
\end{lemma}

\begin{lemma}\label{lemma: second lemma}
For any $m\ge 1$ we have
\begin{gather*}
\sum_{n\ge 0}\left[\left((\mbC^2)^{[n]}\right)^{T_{0,m}\times\Gamma_m}\right]t^n=\prod_{\substack{i\ge 1\\m\nmid i}}\frac{1}{1-t^i}\prod_{i\ge 1}\frac{1}{1-\mbL t^{m i}}.
\end{gather*}
\end{lemma}

\begin{proof}[Proof of Lemma \ref{lemma: first lemma}]
Let $\left(\Hilb{n}\right)^{\Gamma_{\alpha+\beta}}=\coprod_i\left(\Hilb{n}\right)^{\Gamma_{\alpha+\beta}}_i$ be the decomposition in the irreducible components. It is sufficient to prove that
\begin{gather*}
\left[\left(\Hilb{n}\right)^{\Gamma_{\alpha+\beta}}_i\right]=\mbL^{\frac{d_i}{2}}\left[\left(\left(\Hilb{n}\right)^{\Gamma_{\alpha+\beta}}_i\right)^{T_{\alpha,\beta}}\right],
\end{gather*}
where $d_i=\dim\left(\Hilb{n}\right)^{\Gamma_{\alpha+\beta}}_i$. The subvarieties $\left(\Hilb{n}\right)^{\Gamma_{\alpha+\beta}}_i$ are quiver varieties of affine type $\tilde A_{\alpha+\beta-1}$. We prove the above equality by using the idea in \cite[Proposition 5.7]{Nak2}.

Let $\left(\left(\Hilb{n}\right)^{\Gamma_{\alpha+\beta}}_i\right)^{T_{\alpha,\beta}}=\coprod_j\left(\Hilb{n}\right)^{\Gamma_{\alpha+\beta}\times T_{\alpha,\beta}}_{i,j}$ be the decomposition in the irreducible components. Consider the $\mbC^*$-action on $\Hilb{n}$ induced by the homomorphism $\mbC^*\to\torus{2}, t\mapsto(t^\alpha,t^\beta)$. Define the sets $C_{i,j}$ by 
$$
C_{i,j}=\left\{\left.z\in\left(\Hilb{n}\right)^{\Gamma_{\alpha+\beta}}_i\right|\lim_{t\to 0,t\in\mbC^*}t\cdot z\in\left(\Hilb{n}\right)^{\Gamma_{\alpha+\beta}\times T_{\alpha,\beta}}_{i,j}\right\}.
$$
From \cite{B-B1,B-B2} it follows that the set $C_{i,j}$ is a locally trivial fiber bundle over $\left(\Hilb{n}\right)^{\Gamma_{\alpha+\beta}\times T_{\alpha,\beta}}_{i,j}$ with an affine space as a fiber. Let us denote by $d_{i,j}$ the dimension of a fiber. For $p\in \left(\Hilb{n}\right)^{\Gamma_{\alpha+\beta}\times T_{\alpha,\beta}}_{i,j}$ the tangent space $T_p\left(\Hilb{n}\right)^{\Gamma_{\alpha+\beta}}_i$ is a $\mbC^*$-module. Let 
$$
T_p\left(\Hilb{n}\right)^{\Gamma_{\alpha+\beta}}_i=\sum_{m\in\mbZ}H(m)
$$
be the weight decomposition. It is clear that $d_{i,j}=\dim\left(\bigoplus_{m\ge 1} H(m)\right)$. 

The Hilbert scheme $\Hilb{n}$ has the canonical symplectic form $\omega$ that is induced from the symplectic form $dx\wedge dy$ on $\mbC^2$ (see~e.g.\cite{Nak1}). The form $\omega$ has weight $-\alpha-\beta$ with respect to the $\mbC^*$-action on $\Hilb{n}$. The restriction $\omega|_{\left(\Hilb{n}\right)^{\Gamma_{\alpha+\beta}}_i}$ is the canonical symplectic form on the quiver variety (see~\cite{Nak2}). Therefore, the spaces $\bigoplus_{m\le 0}H(m)$ and $\bigoplus_{m\ge\alpha+\beta}H(m)$ are dual with respect to this form. Obviously, the $(\alpha+\beta)$-th root of unity $\sqrt[\alpha+\beta]{1}$ acts trivially on $\left(\Hilb{n}\right)^{\Gamma_{\alpha+\beta}}_i$, thus, $H(m)=0$, if $(\alpha+\beta)\nmid m$. We get $\bigoplus_{m\ge\alpha+\beta}H(m)=\bigoplus_{m\ge 1} H(m)$ and $d_{i,j}=\dim\left(\bigoplus_{m\ge 1}H(m)\right)=\frac{d_i}{2}$. This completes the proof of the lemma.
\end{proof}

\begin{proof}[Proof of Lemma \ref{lemma: second lemma}]
Obviously, we have $\left(\Hilb{n}\right)^{T_{0,m}}=\left(\Hilb{n}\right)^{T_{0,1}}$. For a partition $\lambda=(\lambda_1,\ldots,\lambda_l)$, $\lambda_1\ge \lambda_2\ge\ldots\ge\lambda_l\ge 1$, and a point $x_0\in\mbC$ define the ideal $I_{\lambda,x_0}\subset\mbC[x,y]$ by
\begin{gather*}
I_{\lambda,x_0}=(y^{\lambda_1},(x-x_0)y^{\lambda_2},\ldots,(x-x_0)^{l-1}y^{\lambda_l},(x-x_0)^l).
\end{gather*}
In \cite{Nak1} it is proved that each element $I\in\left(\Hilb{n}\right)^{T_{0,1}}$ can be uniquely expressed as
$$
I=I_{\lambda^1,x_1}\cap\ldots\cap I_{\lambda^k,x_k}
$$
for some distinct points $x_1,\ldots,x_k\in\mbC$ and for some partitions $\lambda^1,\ldots,\lambda^k$ satisfying $\sum_{i=1}^k|\lambda^i|=n$.

Denote by $\mbC_x$ the $x$-axis in the plane $\mbC^2$. Consider the map $\pi_n\colon \left(\Hilb{n}\right)^{T_{0,1}}\to S^n\mbC_x$ defined by
$$
\pi_n\left(I_{\lambda^1,x_1}\cap\ldots\cap I_{\lambda^k,x_k}\right)=\sum_{i=1}^k|\lambda^i|[x_i].
$$ 

Suppose $Z$ is an open subset of $\mbC_x$. From \eqref{formula: classformula} it follows that
$$
\sum_{n\ge 0}\left[\pi_n^{-1}\left(S^n Z\right)\right]t^n=\left(\prod_{i\ge 1}\frac{1}{1-t^i}\right)^{[Z]}.
$$

The $\Gamma_m$-action on $\mbC_x\backslash\{0\}$ is free and $(\mbC_x\backslash\{0\})/\Gamma_m\cong\mbC_x\backslash\{0\}$, therefore,
\begin{gather*}\label{formula: 2}
\left(\pi_n^{-1}\left(S^n(\mbC_x\backslash\{0\})\right)\right)^{\Gamma_m}\cong
\begin{cases}
\emptyset,&\text{if $m\nmid n$},\\
\pi_l^{-1}\left(S^l(\mbC_x\backslash\{0\})\right),&\text{if $n=m l$}.
\end{cases}
\end{gather*} 
We obtain
$$
\sum_{n\ge 0}\left[\left(\pi_n^{-1}\left(S^n(\mbC_x\backslash\{0\})\right)\right)^{\Gamma_m}\right]t^n=\left(\prod_{i\ge 1}\frac{1}{1-t^{m i}}\right)^{\mbL-1}.
$$
Therefore, we get
\begin{multline*}
\sum_{n\ge 0}\left[\left((\mbC^2)^{[n]}\right)^{T_{0,1}\times\Gamma_m}\right]t^n=\left(\sum_{n\ge 0}[\pi_n^{-1}(n[0])]t^n\right)\left(\sum_{n\ge 0}\left[\left(\pi_n^{-1}\left(S^n(\mbC_x\backslash\{0\})\right)\right)^{\Gamma_m}\right]t^n\right)=\\
=\left(\prod_{i\ge 1}\frac{1}{1-t^i}\right)\left(\prod_{i\ge 1}\frac{1}{1-t^{mi}}\right)^{\mbL-1}=\prod_{\substack{i\ge 1\\ m\nmid i}}\frac{1}{1-t^i}\prod_{i\ge 1}\frac{1}{1-\mbL t^{mi}}.
\end{multline*}
The lemma is proved.
\end{proof}

The theorem is proved.

\subsection*{\it Acknowledgements}

A. B. is partially supported by grant ERC-2012-AdG-320368-MCSK in the group of R. Pandharipande at ETH Zurich, by a Vidi grant of the Netherlands Organization of Scientific Research, by Russian Federation Government grant no. 2010-220-01-077 (ag. no. 11.634.31.0005), by the grants RFBR-10-01-00678, NSh-4850.2012.1, the Moebius Contest Foundation for Young Scientists and by ``Dynasty'' foundation. Research of B. F. was carried out within "The national Research University Higher School of Economics" Academic Fund Program in 2013-2014, research grant 12-01-0016. H. N. is supported by the Grant-in-aid for Scientific Research (No.23340005), JSPS, Japan.

\end{document}